\documentclass{amsart}

\usepackage[centertags]{amsmath}
\usepackage{amssymb,amsthm}
\usepackage[all]{xy}
\usepackage{hyperref}
\usepackage{graphicx}

\title{On the $h$-cobordism category. I}

\author{George Raptis}
\address{\newline 
G. Raptis \newline 
Fakult\"{a}t f\"{u}r Mathematik, Universit\"{a}t Regensburg, D-93040 Regensburg, Germany}
\email{georgios.raptis@mathematik.uni-regensburg.de}

\author{Wolfgang Steimle}
\address{\newline
W. Steimle \newline 
Institut f\"ur Mathematik,
Universit\"at Augburg,
D-86135 Augsburg, Germany}
\email{wolfgang.steimle@math.uni-augsburg.de}

\DeclareMathOperator{\colim}{colim}

\DeclareMathOperator{\Diff}{Diff}
\DeclareMathOperator{\Emb}{Emb}

\DeclareMathOperator*{\hocolim}{hocolim}

\DeclareMathOperator{\id}{id}

\DeclareMathOperator{\mor}{mor}

\DeclareMathOperator{\Wh}{Wh}

\DeclareMathOperator{\Imm}{Imm}
\DeclareMathOperator{\inter}{int}
\DeclareMathOperator{\Bun}{Bun}

\newcommand{\cobx}{\mathcal C}

\newcommand{\cobd}{\mathcal{C}_{d}}
\newcommand{\cobdn}{\mathcal{C}_{d, n}}
\newcommand{\cobdstr}{\mathcal{C}_\theta}
\newcommand{\cobdnstr}{\mathcal{C}_{\theta,n}}

\newcommand{\corn}[2]{\RR^{#1}_{\langle #2\rangle}}

\newcommand{\ob}{\operatorname{ob}}

\newcommand{\eps}{\varepsilon}

\newcommand{\inv}{^{-1}}

\newcommand{\hcobd}{\mathcal{C}_{\theta}^{\sim}}
\newcommand{\hcobdno}{\mathcal{C}_d^{\sim}}

\newcommand{\dell}{\partial}

\newcommand{\RR}{\mathbb{R}}

\newcommand{\T}{\mathcal{T}_M}

\newcommand{\thetah}{h_{\theta}}

\newcommand{\Hdiff}{\mathcal{H}}

\begin{document}

\theoremstyle{plain}
\newtheorem{thm}{Theorem}
\newtheorem{cor}[thm]{Corollary}
\newtheorem{lem}[thm]{Lemma}
\newtheorem{prop}[thm]{Proposition}
\newtheorem{claim}[thm]{Claim}
\theoremstyle{definition}
\newtheorem{defn}[thm]{Definition}
\newtheorem{obs}[thm]{Observation}
\newtheorem{constr}[thm]{Construction}

\theoremstyle{remark}
\newtheorem{rem}[thm]{Remark}

\setcounter{secnumdepth}{2}
\numberwithin{thm}{section}

\SelectTips{eu}{10}
\renewcommand{\theenumi}{\roman{enumi}}
\renewcommand{\labelenumi}{\textup{(}\theenumi\textup{)}}
\renewcommand{\theequation}{\arabic{equation}}

\begin{abstract}
We consider the topological category of $h$-cobordisms between smooth manifolds with boundary and compare its homotopy type 
with the standard $h$-cobordism space of a compact smooth manifold. 
\end{abstract}

\maketitle
\setcounter{tocdepth}{1}
\tableofcontents

\section{Introduction}

The spaces of $d$-dimensional smooth cobordisms between closed manifolds define a topological category  where the composition is given by concatenating cobordisms. The homotopy type of the 
classifying space of this cobordism category was identified with the infinite loop space of a Thom spectrum \cite{GMTW(2009)}. Analogous results have been obtained for variants of this cobordism category, 
defined either by considering general tangential structures, or by allowing cobordisms between manifolds with boundaries, corners, or Baas-Sullivan singularities (see, for example, \cite{Genauer(2008)} and \cite{Perlmutter}).  These 
results have turned out to be of fundamental importance in the study of diffeomorphism groups, characteristic classes, invertible field theories, and positive scalar curvature metrics. 

The composition of two $h$-cobordisms is again an $h$-cobordism, so the spaces of $d$-dimensional smooth $h$-cobordisms form a subcategory of the standard cobordism category. In this 
paper, we study the homotopy type of (the loop space of) this $h$-\emph{cobordism category}, in the setting of smooth manifolds with boundary, and compare this with the standard 
(stable) parametrized $h$-cobordism space. The latter space is closely connected with the algebraic $K$-theory of spaces as a consequence of the celebrated stable parametrized $h$-cobordism theorem \cite{WJR}.

\medskip

Recall that if $M^{d-1}$ and $N^{d-1}$ are compact smooth $(d-1)$-manifolds with boundary, a \emph{cobordism} from $M$ to $N$ is a compact smooth $d$-manifold $W^d$ with boundary, together with a decomposition of its boundary 
into three submanifolds of codimension $0$,
\[\dell W = M \cup \dell_h W \cup N,\]
such that $M$ and $N$ are disjoint and $M \amalg N$ meets $\dell_h W$ in the common boundary $\dell M \amalg \dell N = \dell (\dell_h  W)$. Such a cobordism $W$ is called an \emph{$h$-cobordism} if 
each of the inclusions
\[M \hookrightarrow W \hookleftarrow N, \quad  \dell M \hookrightarrow \dell_h  W \hookleftarrow \dell N\]
is a  homotopy equivalence. As  usual, there is a generalization to the case where the manifolds are endowed with tangential $\theta$-structures. This leads to the definition of the 
\emph{$\theta$-structured  $h$-cobordism category}, denoted by $\hcobd$.  (See section \ref{sec:h_cob_cat} for the details. For notational simplicity, we suppress from the notation any
symbol to indicate the existence of ``horizontal'' boundaries.) 

We emphasize that in our definition the horizontal boundary $\dell_h W$  can be an arbitrary $h$-cobordism between the closed manifolds $\dell M$ and $\dell N$. This is in contrast to 
the $h$-cobordisms considered in the definition of the standard $h$-cobordism space $\Hdiff(M)$,  where one has 
$\dell_h W = \dell M\times [0,1]$ and therefore $\dell M = \dell N$. An $h$-cobordism of this special form will be called a \emph{$\dell$-trivial $h$-cobordism} in this paper. The 
additional flexibility of our definition will be of great importance for the comparison between the classifying space of the $h$-cobordism category and the $h$-cobordism space. More precisely, the 
$h$-cobordism space $\Hdiff(M)$ of a compact smooth $(d-1)$-manifold $M$ is the classifying space for bundles of $\dell$-trivial $h$-cobordisms which agree with the trivial 
$M$-bundle on one side of the parametrized family of $\partial$-trivial $h$-cobordisms, while the restriction to the other side is an arbitrary bundle of compact smooth manifolds (see, for example, \cite{WJR}). Our main result compares, 
in the case $d\geq 7$, the  homotopy type of $B\hcobd$ with the homotopy types of $\Hdiff(M)$ for different $M$.

\begin{thm}\label{thm:main1}
Let $M^{d-1}$ be an object of $\hcobd$, where $d\geq 7$. For any $\dell$-trivial $h$-cobordism $(W;M,N)$ on $M$, there is a homotopy fiber sequence
\[\Emb_\theta^\sim(M, N) \to \Omega_M B\hcobd \to \Hdiff(M)\]
where the loop space is based at $M$, and the homotopy fiber is taken over $(W; M, N)$ in $\Hdiff(M)$.
\end{thm}

Here $\Emb_\theta^\sim(M, N)$ denotes the \emph{space of $\thetah$-embeddings} from $M$ into $N$. These are embeddings $\iota \colon M \to \inter(N)$ which have the property that both inclusions 
of $\iota(\dell M)$ and $\dell N$ into the complement $N-\inter\big(\iota(M)\big)$ are homotopy equivalences (so that, in particular, the embedding itself is a homotopy equivalence). 
Moreover, an identification is required between the original $\theta$-structure on $M$ and the one induced from $N$ through the embedding $\iota$, where $N$ is given the $\theta$-structure inherited from 
$M$ through the $h$-cobordism $W$. See section \ref{sec:h_cob_space} for the precise definition. 

\medskip

An interesting consequence of Theorem \ref{thm:main1} is in the case where the tangent bundle of $M$ is regarded as a $\theta$-structure and $M$ is a compact smooth $(d-1)$-manifold 
of lower handle dimension. Let $\theta_M = (\eps \oplus TM \to M)$, where $\eps$ denotes the trivial line bundle.

\begin{thm}\label{thm:main2}
Let $M$ be a compact connected smooth $(d-1)$-manifold of handle dimension $k$, regarded as an object of $\mathcal{C}^{\sim}_{\theta_M}$, where $d \geq 7$. Then there is a $(d-2k-2)$-connected map
\[
\Omega_M B\mathcal{C}^{\sim}_{\theta_M} \xrightarrow{(d-2k-2)-con} \Hdiff(M).
\]
\end{thm}

To draw our final conclusion, we note that the domain of the map in Theorem \ref{thm:main2} may be stabilized by taking product with $[0,1]$ on objects and on morphisms, and then 
straightening the resulting corners (as discussed in the Appendix). This is compatible with the usual stabilization of the $h$-cobordism space in the target of the map. Therefore we conclude the following corollary as a consequence of 
Theorem \ref{thm:main2} and the stable parametrized $h$-cobordism theorem \cite{WJR}.

\begin{cor}\label{thm:cor}
Let $M$ be a compact connected smooth $(d-1)$-manifold. There is a homotopy equivalence
\[\hocolim_n \Omega_{M \times D^n} B\mathcal{C}^{\sim}_{\theta_{M \times D^n}} \simeq \Omega \Wh^{\operatorname{diff}}(M).\]
\end{cor}

Here $\Wh^{\operatorname{diff}}(M)$ denotes the smooth Whitehead space  of $M$, defined by means of algebraic $K$-theory (see \cite{WJR}). We note here that we work entirely in the 
smooth setting in this paper. Finally, we remark that we do not know a direct description of the resulting infinite--loop space structure on the left--hand side of the homotopy equivalence in Corollary \ref{thm:cor}. (Note that $B\mathcal{C}^{\sim}_\theta$ is an $E_\infty$-monoid under disjoint union, but it is typically not group--like.)

\medskip 

The strategy for the proof of these results is as follows. First, using the classical $h$-cobordism theorem, the category $\hcobd$ is shown to be a (non-unital) groupoid in a homotopical 
sense (Section \ref{sec:invertibility_hcob}). This is indeed the only place where we use the dimension restriction -- in particular, our techniques apply in any dimension if we restrict to the cobordism category of \emph{invertible} cobordisms instead. 
Using standard homotopical techniques discussed in Section \ref{sec:non_unital_categories}, it follows that the classifying space $B \hcobd$ is homotopy equivalent to a disjoint union of spaces of endomorphisms, one 
from each component. In Section \ref{sec:h_cob_space}, we analyze geometrically the homotopy type of an endomorphism space in $\hcobd$. The main idea is to trade the $\dell$-triviality of the $h$-cobordisms in $\Hdiff(M)$ with the condition that each $h$-cobordism in an endomorphism space has both ends fixed. This analysis leads to Theorem \ref{thm:main1}. Theorem \ref{thm:main2} follows by analyzing 
the homotopy types of the $\thetah$-embedding spaces under consideration. 

\medskip

In a sequel to this work, we aim to carry out a different analysis of the homotopy type of $B\hcobd$ and study directly its connection with algebraic $K$-theory.

\medskip 

\noindent \textbf{Acknowledgements.} The first named author was partially supported by the \emph{SFB 1085 -- Higher Invariants} (University of Regensburg) funded by the DFG. The second named author was partially supported by the DFG through the SPP 2026 ``Geometry at infinity".

\section{The $h$-cobordism category}\label{sec:h_cob_cat}

In this section, we recall the definition of the cobordism category $\cobdstr$ \cite{GMTW(2009), Genauer(2008)} and then define the $h$-cobordism category $\hcobd$ to be the subcategory of $h$-cobordisms. The definition 
of the cobordism category here uses a slight variation of the definition of tangential structure as defined in  \cite{GMTW(2009)} and \cite{Genauer(2008)}. We emphasize that \emph{objects (morphisms) 
are allowed to have boundaries (corners)}. Moreover, we will need to allow arbitrary tangential structures as these will play an important role in the proofs of our main results.

\medskip

Let $\RR_+:=[0,\infty)$, and $\corn{n}{k}:= \RR_+^k \times \RR^{n-k}\subset \RR^n.$

We call a subset $M\subset \corn{d-1+n}1$ a $(d-1)$-\emph{dimensional neatly embedded submanifold} if the following hold:
\begin{enumerate}
 \item $M\subset \corn{d-1+n}1$ is locally diffeomorphic to an open subset of $\corn{d-1}1\subset \corn{d-1+n}1$. In particular, $\dell M = M \cap \dell\corn{d-1+n}1$.
 \item There is $\epsilon>0$ such that
\[M \cap \bigl([0, \epsilon) \times \RR^{d-2+n}\bigr)= [0, \epsilon) \times \dell M,\]
\end{enumerate}
with the appropriate interpretation of the product on the right--hand side as a subspace of $\corn{d-1+n}1$.

\medskip

Fix a vector bundle $\theta = (V \to X)$ of rank $d$ over a space $X$. A \emph{$\theta$-structure} on a smooth $(d-1)$-dimensional manifold $M$ (possibly with boundary) 
is a map of vector bundles $l_M\colon \eps\oplus TM\to \theta$, that is, a map between the total spaces over a map of base spaces which restricts to a linear isomorphism 
in each fiber. Here $\eps$ denotes the trivial line bundle. 

\medskip

Note that the boundary of the manifold $[0,a]\times \corn n 1$ consists of three boundary faces, which we denote by
\begin{align*}
\dell_0([0,a]\times \corn n 1) & := \{0\}\times \corn n 1,\\
\dell_1([0,a]\times \corn n 1) &:= \{a\}\times \corn n 1,\\
\dell_h([0,a]\times \corn n 1) &:= [0,a]\times \dell \corn n 1.
\end{align*}

We call a subset $W\subset [0,a]\times \corn{d-1+n}1$ a $d$-\emph{dimensional neatly embedded submanifold} if the following hold:
\begin{enumerate}
 \item $W\subset [0,a]\times \corn{d-1+n}1$ is locally diffeomorphic to an open subset of $\corn d 2 \subset \corn {d+n}2$. In particular, $W$ is a smooth $d$-manifold, possibly with corners. 
 \item There is $\epsilon>0$ such that
\begin{enumerate}
\item $W \cap ([0, \epsilon) \times \corn{d-1 + n}{1}) = [0, \epsilon) \times \dell_0 W$,
\item $W \cap ((a - \epsilon, a] \times \corn{d-1 + n}{1}) = (a - \epsilon, a] \times \dell_1 W$,  and
\item $W \cap ([0, a] \times [0, \epsilon) \times \RR^{d-2+n})= [0, \epsilon) \times \partial_h W$.
\end{enumerate}
with the appropriate interpretation of the products on the right. Here we have used the notation 
$$\dell_i W :=  W\cap \dell_i([0,a]\times \corn{d-1 + n}1), \quad (i\in \{0,1,h\}).$$ 
\end{enumerate}
We call $\dell_h W$ the \emph{horizontal} boundary of $W$ and $\dell_0  W\amalg \dell_1 W$ the \emph{vertical} boundary of $W$.

\medskip

A $\theta$-structure on a smooth $d$-dimensional manifold $W$ (possibly with corners) is a map of vector bundles $l_W\colon TW\to \theta$. Note that such 
a $\theta$-structure induces a $\theta$-structure $l_{\dell_i W}$ on $\dell_i W$, $i\in \{0,1,h\}$, by means of the canonical splittings
\begin{align*}
 TW\vert_{\dell_0 W} &= T[0, \epsilon)\vert_{0} \times T(\dell_0 W)  \cong \eps\oplus T(\dell_0 W),\\
  TW\vert_{\dell_1 W} &= T(a - \epsilon, a]\vert_{a} \times T(\dell_1 W) \cong \eps\oplus T(\dell_1 W),\\
  TW\vert_{\dell_h W} &= T[0, \epsilon)\vert_{0} \times T(\dell_h W)  \cong \eps\oplus T(\dell_h W),
\end{align*}
induced by the standard trivialization of $T\RR$, sending $\dell/\dell x$ to 1.

\begin{defn} Let $\theta = (V \to X)$ be a vector bundle on a space $X$ of rank $d$. The \emph{$\theta$-cobordism category $\cobdnstr$ in $\corn{d+n}1$} is defined as follows:
\begin{itemize}
 \item An object is a pair $(M, l_M)$ where $M\subset \corn{d-1+n}1$ is a compact $(d-1)$-dimensional neatly embedded submanifold; $l_M$ is a $\theta$-structure on $M$.
 \item A morphism is a triple $(a, W, l_W)$ where $a>0$ is a real number;  $W\subset [0, a]\times \corn{d-1+n}1$ is a compact $d$-dimensional neatly embedded submanifold; 
 $l_W$ is a $\theta$-structure on $W$. 
 \item The source and target of a morphism $(a,W, l_W)$ are given by $(\dell_i W, l_{\dell_i W})$ for $i=0$ and $i=1$, respectively.
 \item Composition of morphisms is defined by 
 \[(a', W', l') \circ (a, W, l) := (a+a', W\cup \mathrm{sh}_a(W'), L)\]
where $\mathrm{sh}_a(W')$ denotes the embedding of $W'$ shifted by $\mathrm{sh}_a \colon \RR \to \RR$, $x \mapsto x + a$, in the first coordinate; $L$ is defined by the conditions $l\vert_{W}= l$ and 
$l\vert_{\mathrm{sh}_a(W')} = l' \circ D\mathrm{sh}_a\inv$. 
\end{itemize}
\end{defn}

We use the abbreviation $\cobdstr :=  \underset{n \to \infty}{\colim} \ \cobdnstr.$ In this case, note that $n$ is arbitrarily large and not part of the structure.  
We will use the notation $\cobdn$ and $\mathcal{C}_d$, respectively, when no 
$\theta$-structures are considered in the definition of the objects and the morphisms of the category. We will generally regard this as a special case of 
$\cobdstr$ since it corresponds up to homotopy equivalence to the case of the 
universal vector bundle of rank $d$.  

\medskip

Next we recall the definition of the topology on the spaces of objects. For notational simplicity, we restrict only to the case $n=\infty$. Given  a compact smooth $(d-1)$-manifold $M$ with $\eps$-collared boundary, let $\mathrm{Emb}^\eps(M, \RR_+ \times \RR^{\infty})$ denote the space $\eps$-neat 
embeddings with  the $C^\infty$-topology. Here ``$\eps$-neat embedding'' means a diffeomorphism onto a smooth neatly embedded submanifold, which is cylindrical in the collar 
coordinate inside $[0, \eps) \times \RR^{\infty}$. Then we define   $\mathrm{Emb}(M, \RR_+ \times \RR^{\infty})$ to be the colimit of these spaces as $\eps \to 0$. 
There is a canonical map
\[ \mathrm{Emb}(M, \RR_+\times \RR^{\infty})\times \Bun(\eps\oplus TM, \theta)\to \ob \cobdstr,\]
which sends $(e,l)$ to $(e(M), l\circ (\id\oplus De\inv))$. The topology on the space of objects is defined by the quotient topology with respect to the collection of these maps, 
one for each $M$, where one representative from each diffeomorphism class suffices. 

Similarly, let $\Diff_{\eps}(M)$ denote the group of diffeomorphisms which are cylindrical near the collared boundary, and let $\Diff(M) := \colim_{\eps \to 0} \Diff_{\eps}(M)$. Note 
that $\Diff(M)$ acts freely on $\mathrm{Emb}(M, \RR_+\times \RR^{\infty})\times \Bun(\eps\oplus TM, \theta)$, by precomposition in both factors. We denote by $B_\theta(M)$ the 
quotient of this action; then we have a homeomorphism
\[\ob \cobdstr \cong \coprod_{M} B_\theta(M)\]
where the coproduct ranges over compact smooth $(d-1)$-manifolds $M$ with collared boundary, one from each diffeomorphism class. 

The analogous construction for compact $d$-dimensional cobordisms, possibly with corners, which are collared, defines the topology of the space of morphisms. In this case, we have a homeomorphism
\[\mor\cobdstr \cong \coprod_{W} (0,\infty) \times B_\theta(W)\]
where $B_\theta(W)$ is the quotient under the $\Diff(W)$-action on the product
\[\Emb(W, [0,1]\times \RR_+\times \RR^\infty)\times \Bun(TW, \theta).\]

\begin{rem}
 Our definition of the cobordism category differs from \cite{Genauer(2008)} or \cite{Perlmutter} in several respects. Firstly, we use a different definition of tangential structure which does 
 not require a model for the classifying map of the tangent bundle. Secondly, we use a reduced version of the cobordism category as suggested in \cite[Remark 2.1]{GMTW(2009)}. These technical 
 modifications do not affect the homotopy type of the spaces of objects or morphisms, or of the classifying space. Finally, the category $\cobdnstr$ has no identity morphisms, and we shall
 consider it here as a \emph{non-unital category}, while in \cite{GMTW(2009)} the identity morphisms are formally added. Again, the homotopy type of the classifying space is not affected by this 
 difference (see also Section \ref{sec:non_unital_categories}).
\end{rem}

\begin{defn}
 Let $M$ and $N$ be compact smooth $(d-1)$-manifolds, possibly with boundary. A cobordism (with corners) $W = (W; \partial_v W = M \amalg N, \partial_h W)$ from $M$ to $N$ is an 
 $h$-\emph{cobordism} if all of the inclusions
 \[M\hookrightarrow W\hookleftarrow N, \quad  \dell M \hookrightarrow \dell_h W \hookleftarrow \dell N\]
are  homotopy equivalences. The category $\hcobd\subset \cobdstr$ is the subcategory which has the same objects as $\cobdstr$, and where the morphisms are 
$h$-cobordisms. We will also denote by $\hcobdno \subset \cobd$ the corresponding subcategory where no $\theta$-structures are considered. 
\end{defn}

\section{Non-unital categories and semi-simplicial spaces} \label{sec:non_unital_categories}

Let $\Delta_{<} \subseteq \Delta$ denote the subcategory of the simplex category which consists of the injective maps. A functor $X_{\bullet}: \Delta_{<}^{\mathrm{op}} \to \mathrm{Spaces}$ is a \emph{semi-simplicial space}. Similarly to simplicial spaces, $X_{\bullet}$ has a geometric realization $\|X_\bullet\|$. This is isomorphic to the geometric realization of the simplicial space which is associated to $X_\bullet$ by freely adding degeneracies (i.e., by taking the left 
Kan extension of $X_\bullet$ along the inclusion $\Delta_{<} \subseteq \Delta$.) 
The geometric realization of a semi-simplicial space admits a skeletal filtration
$$\varnothing = \|X_{\bullet}\|_{-1} \subseteq \| X_{\bullet} \|_0 \subseteq \| X_{\bullet} \|_1 \subseteq \cdots \subseteq \| X_{\bullet}\|_n \subseteq \cdots \subseteq \| X_\bullet\| = \bigcup_{n \geq 0} \|X_{\bullet}\|_n$$
which is defined inductively as follows: $\| X_{\bullet}\|_0 : = X_0$ and there are successive pushout squares 
\[
 \xymatrix{
 X_k \times \partial \Delta^{k} \ar[r] \ar[d] & \|X_{\bullet}\|_{k-1} \ar[d] \\
 X_k \times \Delta^k \ar[r] & \|X_{\bullet}\|_k
 }
\]
where the top map is defined using the simplicial operators on $X_{\bullet}$. It follows that the geometric realization of a 
degreewise weak equivalence of semi-simplicial spaces is again a weak equivalence of spaces. We also refer to \cite{ERW} for a detailed treatment of the homotopy theory of semi-simplicial spaces.

\medskip

\noindent The nerve of a non-unital topological category $\cobx$ defines a semi-simplicial space 
$$N_\bullet \cobx: \Delta_{<}^{\mathrm{op}} \to \mathrm{Spaces}, \ [n] \mapsto N_n \cobx := \underbrace{\mor \cobx \times_{\ob \cobx}  \cdots \times_{\ob \cobx} \mor \cobx}_{n}.$$
Note that $N_0 \cobx = \ob \cobx$. The geometric realization of $N_{\bullet} \cobx$ is the \emph{classifying space} of $\cobx$ and 
will be denoted by $B \cobx$. We can associate a unital topological category 
$\cobx \oplus \mathbf{1}$ to a non-unital topological category $\cobx$, simply by formally adding identities (disjointly in the topological sense). Then the classifying space of $\cobx \oplus \mathbf{1}$, as a unital topological category, is isomorphic to the classifying space of $\cobx$, as a non-unital topological category. 

\medskip

The following proposition is a variation of well-known results but formulated in the context of non-unital categories. Given a space $X$ and points $x,y \in X$, we write $\Omega_{x, y} X$ for the space of paths in $X$ starting at $x$ and ending at $y$. If $\cobx$ is a (non-unital, topological) category and $X, Y$ are objects of $\cobx$, 
then every morphism from $X$ to $Y$ defines an element in $\Omega_{X, Y} B \cobx$,
and thus we obtain a continuous map 
$$i_{X,Y}: \cobx(X, Y) \to \Omega_{X,Y} B \cobx.$$
For an object $X$ in a (non-unital, topological) category $\cobx$, there is a `transport' (non-unital, topological) category $X \wr \cobx$ associated to 
the functor which is represented by the object $X$. The objects of 
$X \wr \cobx$ are the morphisms in $\cobx$ with source $X$. This set 
of objects is topologized as a subspace of the space of all morphisms 
in $\cobx$. The morphisms are given by morphisms in $\cobx$ that 
make the obvious triangle commute. The topology is defined by a pullback square of spaces
\begin{equation} \label{transport-cat}
\xymatrix{
\mor{(X \wr \cobx)} \ar[r] \ar[d] & \ob{(X \wr \cobx)} \ar[d] \\
  \mor{(\cobx)} \ar[r] & \ob{(\cobx)}
  } 
 \end{equation}
where the horizontal maps evaluate at the source. (This data can be used as the definition of the transport category, see \cite[pp. 235--236]{GMTW(2009)}.) 

Note that the topological category $X \wr \cobx$ has an initial 
object if $\cobx$ is unital: this is given by the identity of $X$. 
In this case, it follows that its classifying space is contractible.
In general, there is a canonical functor 
$(X \wr \cobx) \oplus \mathbf{1} \to X \wr (\cobx \oplus \mathbf{1})$
which is not an equivalence of categories.
For every object $Y$ in $\cobx$, there is a continuous map 
$$j_{X, Y}: \cobx(X, Y) \to B(X \wr \cobx)$$ 
which is given by the inclusion of $0$-simplices. 

\begin{prop} \label{technical-lemma}
Let $\cobx$ be a non-unital topological category and let $X$ be an object of $\cobx$ such that the projection map 
\[\ob(X\wr \cobx) \to \ob(\cobx)\]
is a Serre fibration, and for every morphism $f: Y \to Z$, the induced map 
$$f_*: \cobx(X, Y) \to \cobx(X, Z)$$
is a weak equivalence of spaces. Then, for any object $Y$ of $\cobx$, the canonical map 
$$(i_{X, Y}, j_{X, Y}): \cobx(X, Y) \to \Omega_{X, Y} B \cobx \times B(X \wr \cobx)$$
is a weak equivalence. 
\end{prop}
\begin{proof}
There is a projection functor $p: X \wr \cobx \to \cobx$, $(X \to Y) \mapsto Y$. The induced map between classifying spaces is homotopic to the constant map at $X$. To see this, it suffices to note that the associated functor 
$$p \oplus \mathbf{1} \colon (X \wr \cobx) \oplus \mathbf{1} \to \cobx \oplus \mathbf{1}$$
admits a natural tranformation from the constant functor at $X$. Thus, the homotopy fiber of $B(X \wr \cobx) \to B\cobx$ at $Y \in \ob(\cobx)$ is canonically identified with $\Omega_{X, Y}(B\cobx) \times B(X \wr \cobx)$. 
Under this identification, the map $(i_{X, Y}, j_{X, Y})$ is then identified using the following square,

\[
 \xymatrix{
   \ob{(X \wr \cobx)} \ar[r] \ar[d] & B(X \wr \cobx) \ar[d] \\
  \ob{(\cobx)} \ar[r] & B(\cobx),
  }
 \]
as the inclusion of the actual fiber of the left map into the homotopy fiber of the right map. Since the left map is a Serre fibration, the assertion will follow if we show that the 
square above is a homotopy pullback.

By applying \cite[Proposition 1.6]{Segal} (see also \cite[Theorem 2.12]{ERW}), it suffices to show that for each face map $\delta_i \colon [k] \to [k+1]$, the square 
\begin{equation} \label{pullback_sq}
\xymatrix{
N_{k+1}(X \wr \cobx) \ar[d] \ar[r]^{d_i} & N_{k}(X \wr \cobx) \ar[d] \\
N_{k+1}(\cobx) \ar[r]_{d_i} & N_{k}(\cobx)
}
\end{equation}
is a homotopy pullback.  Note that the vertical projection $N_{k}(X \wr \cobx) \to N_{k}(\cobx)$ is a Serre fibration since it is obtained as a pullback of 
$\ob(X \wr \cobx) \to \ob(\cobx)$. The square is a pullback square if $i > 0$, therefore also a homotopy pullback. For $i=0$, the induced 
map between vertical (homotopy) fibers at $(Y_0 \xrightarrow{f_1} Y_1 \to \cdots \to Y_{k+1}) \in N_{k+1}(\cobx)$ is given by
$$f_1{}_* \colon \cobx(X, Y_0) \to \cobx(X, Y_1)$$
which is a weak equivalence by assumption. This shows that \eqref{pullback_sq} is a homotopy pullback for every $k \geq 0$ and $0 \leq i \leq k+1$ and
concludes the proof.   
\end{proof}

\begin{rem}
Proposition \ref{technical-lemma} holds more generally for classes of weak equivalences between spaces that are closed under homotopy colimits. 
For example, if the maps $f_*$ are homology equivalences, then the resulting map $(i_{X, Y}, j_{X, Y})$ is also a homology equivalence.
The proof is essentially the same. See also \cite[Appendix A]{Raptis}.
\end{rem}

\section{Basic properties of the $h$-cobordism category} \label{sec:invertibility_hcob}

In this section, we prove some basic results about the homotopy type of the classifying space of the $h$-cobordism category $\hcobd \subset \cobdstr$. The main result (Proposition 
\ref{hcod-is-a-grpd}) shows that the homotopy type of each component reduces to the homotopy type of a space of endomorphisms. This is a consequence of the invertibility of 
$h$-cobordisms by the classical $h$-cobordism theorem, combined with the results of Section \ref{sec:non_unital_categories}.

\medskip

We will need the following preliminary lemmas.

\begin{lem} \label{etale-category}
The source-target map 
\[(s,t): \mor \cobdstr \to \ob \cobdstr \times \ob \cobdstr\]
is a Serre fibration. As a consequence, the same is true for $\hcobd$. 
\end{lem}
\begin{proof}
It suffices to show that for each cobordism $(W; M_0, M_1)$, the map that restricts to the (vertical) boundary
\begin{equation} \label{Serre-fibration}
B_{\theta}(W) \longrightarrow B_{\theta}(M_0) \times B_\theta(M_1)
\end{equation}
is a Serre fibration. Consider the commutative square
$$
\xymatrix{
\Emb(W, [0,1] \times \corn\infty1) \times \Bun(TW,\theta) \ar[r] \ar[d] & B_{\theta}(W) \ar[d]\\
\Emb(M_0 \sqcup M_1, \{0, 1\} \times \corn \infty 1)\times \Bun(\eps \oplus (TM_0\sqcup TM_1), \theta) \ar[r] & B_{\theta}(M_0) \times B_{\theta}(M_1).
}
$$
The left vertical restriction map is a Serre fibration in each factor. This can be shown for the first factor following the methods of \cite{Lima}; for the second factor this is true because we restrict along a cofibration. Secondly, the horizontal surjective 
quotient maps are also Serre fibrations. This can be obtained as a special case of the main result in \cite[Appendix A]{Perlmutter}. Then
it follows that \eqref{Serre-fibration} is a Serre fibration.
\end{proof}

\begin{lem} \label{thin-identities}
Let $M = (M, l_M)$ be an object of $\cobdstr$. Then $B(M\wr\cobdstr)$ and $B(M \wr \hcobd)$ are weakly contractible.
\end{lem}
\begin{proof}
Consider $M\times [0,1]$ as a $\theta$-manifold with the cylindrical $\theta$-structure induced from $l_M$. For any object $N=(N,l_N)$, the composition map
\[\cobdstr(M, N)\to \cobdstr(M, N), \quad W\mapsto W\circ (M\times [0,1])\]
is a weak equivalence. It follows, using Lemma \ref{etale-category}, that the functor 
\[F\colon M \wr \cobdstr \to M \wr \cobdstr\]
induced by precomposition with $M\times [0,1]$ is a weak equivalence on each level of the nerve, and hence on the geometric realization. 
On the other hand, the map $B(F)$ is homotopically constant because $F\oplus \mathbf 1$ receives a natural transformation from the constant functor 
at $M\times [0, 1]$. This shows that $B(M\wr\cobdstr)$ is contractible. The same argument applies to $B(M\wr\hcobd)$. 
\end{proof}

The following proposition establishes one of the main properties of $\hcobd$ as a consequence of the $h$-cobordism theorem. To state it, we denote by $\pi_0 \hcobd$ 
the ordinary category which has the same objects as $\hcobd$ (with the discrete topology), and morphisms 
\[(\pi_0 \hcobd)(M, N):= \pi_0 (\hcobd(M, N)).\]
Note that this turns out to be a unital category, where $M\times [0,1]$ (with the cylindrical $\theta$-structure) is the identity morphism of $M$.

\begin{defn}
A $\theta$-structured $h$-cobordism $W\in \mor \hcobd$ is called \emph{invertible} if it defines an isomorphism in $\pi_0 \hcobd$.
\end{defn}

\begin{prop} \label{hcod-is-a-grpd}
For $d\geq 7$, every morphism in $\hcobd$ is invertible.  
\end{prop}
\begin{proof}
A morphism in $\hcobd$ is represented by a compact smooth neatly embedded $d$-manifold with corners, depicted as follows,
\[
\xymatrix{
 \dell(\dell_h W) \ar[r] \ar[d] & \dell_h W \ar[d]\\
 M\amalg N \ar[r] & W  
}
\]
together with a bundle map $l_W\colon TW\to \theta$ inducing $\theta$-structures $l_M$ on the source $\dell_0 W=M$ and $l_N$ on the target $\dell_1 W=N$.

It is enough to show that any such morphism admits a right inverse in $\pi_0\hcobd$. We first construct an (abstract) $h$-cobordism $U$ with domain $N$ and target $M$, such  that $W\circ U$ is diffeomorphic, relative to both  ends, to $N\times [0,1]$.
In the case where $\dell_h W$ is empty or diffeomorphic to a product cobordism, and $d\geq 6$, the existence of such $U$ is a well-known consequence of the $h$-cobordism theorem.

To construct $U$ in the general case, we first construct a cobordism $U'$ with target $M$ such that the horizontal boundary $\dell_h(W\circ U')$ is diffeomorphic to a product cobordism. To do this, we choose a right inverse $(\dell_h W)\inv$ of $\dell_h W$. Consider the manifold with corners
$$Z:= [0,1] \times M$$ 
and  choose an embedding of $(\dell_h W)\inv$ into $(0,1] \times \dell M$. After rearranging the boundary pieces of $Z$, we obtain a cobordism $U'$ which ends at $\{1\} \times M$ 
and starts from 
$$\dell_0 U' = \{0\} \times M \cup \overline{([0,1] \times \dell M - (\dell_h W)\inv)},$$
with horizontal boundary $\dell_h U' = (\dell_h W)\inv$. This is again an 
$h$-cobordism. Since  $\dell_h(W \circ U')$ is diffeomorphic to a product cobordism, $W\circ U'$ has a right inverse $U''$. Then, $U:=U'\circ U''$ and $W \circ U$ will be
diffeomorphic, relative to both ends, to $N \times [0,1]$. This completes the  construction of $U$. 

Next choose a collar for $U$ and a neat embedding of $U$ into $[0,1]\times \corn{\infty}1$; this is possible because the space of neat embeddings is contractible \cite[Theorem 2.7]{Genauer(2008)}. 
For the same reason, the induced embeddings of $M\cong \dell_1 U$ and $N\cong \dell_0 U$ into $\corn{\infty}1$ will be isotopic to the inclusion maps. We have already used in the proof of Lemma \ref{etale-category} that the restriction map
\[\Emb(U, [0,1]\times \corn{\infty}1)\to \Emb(N\sqcup M, \{0,1\}\times \corn{\infty}1)\]
is  a Serre fibration; hence after possibly changing the neat embedding of $U$ by an isotopy, we can assume that $\dell_1 U=M$ and $\dell_0 U = N$ as subsets of $\corn{\infty}1$. 

Next we construct a suitable $\theta$-structure on $U$. By the $h$-cobordism condition, any bundle map $\eps\oplus TM\to \theta$ extends to a bundle map $l'_U$ on $TU$, inducing  
a $\theta$-structure $l'_N$ on $N=\dell_0 U$. The composite $\theta$-structure on $W\circ U\cong N\times [0,1]$ is then a homotopy of bundle maps between $l_N$ and $l'_N$. 
We insert a backwards copy of this bundle homotopy in a cylinder near $N\times \{0\}$. This defines a new $\theta$-structure $l_U$ on $U$ which induces $l_N$ on $N=\dell_0 U$. Composing 
$l_U$ with $l_W$ defines a $\theta$-structure on $W \circ U \cong N \times [0,1]$ which is homotopic, relative to both ends, to the cylindrical $\theta$-structure. 

Then the morphism $U=(1, U, l_U)$ is a morphism from $(N, l_N)$ to $(M, l_M)$, and $W \circ U=N\times [0,1]$ in $\pi_0 \hcobd(N,N)$.
\end{proof}

We can now prove our main result in this section. 

\begin{prop} \label{h-cob-cat} 
Let $d \geq 7$ and let $M$ and $N$ be objects of $\hcobd$. Then the canonical map
$$i_{M,N}: \hcobd(M, N) \to \Omega_{M, N} B \hcobd$$
is a weak equivalence.
\end{prop}
\begin{proof}
We check that the assumptions of Proposition \ref{technical-lemma} are satisfied. The projection $\ob (M\wr\hcobd) \to \ob(\hcobd)$ is a Serre fibration because the source-target map $(s,t)$ is a Serre fibration by Lemma \ref{etale-category}.
Moreover, by Lemma \ref{thin-identities}, the space  $B(M \wr \hcobd)$ is weakly contractible. Lastly, any  $h$-cobordism $W\colon N\to P$ in $\hcobd$ is invertible by Lemma \ref{hcod-is-a-grpd}, and it follows easily that the induced map
\[W_*\colon \hcobd(M, N)\to \hcobd(M,P)\]
is a homotopy equivalence of spaces, with right (left) homotopy inverse induced by a right (left) inverse of $W$. 
\end{proof}

\section{Comparison with the $h$-cobordism space}\label{sec:h_cob_space}

\subsection{\normalfont \textit{Outline}}  \label{outline}

We start by giving an informal definition of the terms appearing in Theorem \ref{thm:main1}. We fix a vector bundle $\theta = (V \to X)$ of rank $d$. 
Given $(d-1)$-dimensional compact smooth manifolds $M$, $N$ with $\theta$-structures $l_M$ and $l_N$, a \emph{$\theta$-embedding} from $M$ into $N$ is an embedding $\iota \colon M\to \inter(N)$ 
together with a bundle homotopy from $l_M$ to $l_N \circ (\id \oplus D(\iota))$. It is called an \emph{$\thetah$-embedding} if, moreover, the complement $N- \inter\big(\iota(M)\big)$ is 
an $h$-cobordism from $\iota(\dell M)$ to $\dell N$. 

Given a $\dell$-trivial $h$-cobordism $(W;M,N)$ and a $\theta$-structure on $M$, we first note that $W$ and hence $N$ inherit $\theta$-structures from the  one on $M$. Then given 
an $\thetah$-embedding $\iota \colon M \to \inter(N)$, we can construct an endomorphism of $M$ in $\hcobd$ by introducing corners in $W$ at $\dell M = \dell N$ and at $\iota(\dell M)$. This construction 
defines the map (which depends on $W$)
\[\Emb^\sim_\theta(M,N)\to  \hcobd(M,M) \xrightarrow[\text{Prop.} \ref{h-cob-cat}]{\simeq} \Omega_M B\hcobd\]
that appears in the statement of Theorem \ref{thm:main1}.

On the other hand, given any morphism $W\colon M\to N$ in $\hcobd$, after straightening the corners and forgetting the $\theta$-structure, we may view $W$ as an $h$-cobordism on $M$, 
whose free end is $\dell_h W\cup N$. Applying this construction to the case $N=M$ defines the second map of Theorem \ref{thm:main1},
\[\Omega_M B\hcobd \xleftarrow[\text{Prop.} \ref{h-cob-cat}]{\simeq}  \hcobd(M,M)  \to \Hdiff(M).\]

\medskip

In order to analyze the homotopy fiber of this map $\hcobd(M,M)\to \Hdiff(M)$, we first introduce a common setting in which to compare the given data. This is done in subsection \ref{subsec:convenient_models} where convenient models for these two spaces are constructed in terms of \emph{cornered $\partial$-trivial $h$-cobordisms}. In the case without a tangential structure, both spaces are disjoint unions of classifying spaces of certain diffeomorphism groups, so the homotopy fiber is 
a disjoint union of (homotopy) quotients; this space can be directly compared to the space of $h$-embeddings through the action of the diffeomorphism groups. This is explained in subsection 
\ref{subsec:without_theta} and proves Theorem \ref{thm:main1} in the case without a tangential structure. In subsection \ref{subsec:with_theta}, we conclude from this the statement of 
Theorem \ref{thm:main1} in the general case. Finally, subsection \ref{subsec_proof_of_thm2} is devoted to the proof of Theorem \ref{thm:main2}.

\newcommand{\tT}{\widetilde{\mathcal{T}}_M}
\newcommand{\R}{\mathcal R_{M, N}}
\newcommand{\tR}{\widetilde{\mathcal R}_{M, N}}
\newcommand{\Rfib}{\mathcal R_{M, N}^W}

\newcommand{\hq}{/\!/}
\newcommand{\qh}{\backslash \!\backslash}

\subsection{\normalfont \textit{The convenient models}}\label{subsec:convenient_models} 
Let $M$ be a compact smooth $(d-1)$-manifold with boundary $\partial M$. A \emph{$\dell$-trivial $h$-cobordism on $M$} is a compact smooth 
$d$-manifold $W$ whose boundary splits as a union 
$$\dell W = M\cup N$$ 
along a smooth submanifold of codimension one which is the common boundary of $M$ and $N$, such that both inclusions $M\to W \leftarrow N$ are homotopy equivalences. We call $N \subset W$ the \emph{free boundary part} of the $\partial$-trivial $h$-cobordism $W$.  
Following \cite{WJR}, the \emph{$h$-cobordism space} $\Hdiff(M)_{\bullet}$ is a simplicial set where a 
$p$-simplex is a smooth bundle $\pi \colon E \to \Delta^p$ with a trivial sub-bundle $M \times \Delta^p \subseteq E$ such that for each $x \in \Delta^p$, $\pi^{-1}(x)$ is a 
$\partial$-trivial $h$-cobordism on $M \cong M \times \{x\}$. This is a classifying space for bundles of $\partial$-trivial $h$-cobordisms on $M$, i.e., there is a weak equivalence
\begin{equation} \label{h-cob_space} 
|\Hdiff(M)_{\bullet}| \simeq \coprod_{W\in\tT} B \Diff(W; M) 
\end{equation}
where $\tT$ is a set of $\partial$-trivial $h$-cobordisms on $M$ containing precisely one element from each diffeomorphism class (relative to $M$), and $\Diff(W;M)$ is the topological 
group of diffeomorphisms of $W$ that are the identity near $M$ and are cylindrical in a fixed chosen collar of $\dell W \subset W$. 

\medskip

We may change the smooth structure of such a manifold $(W; M, N) \in \tT$, by introducing codimension 2 corners at the boundary of $M$. In more detail, we start by choosing a neat embedding of $W$ into $\corn{n+d}{1}=\RR_+\times \RR \times \RR^{n+d-2}$ such that $M$ is the part lying in $\{0\}\times \RR_+\times \RR^{n+d-2}$. Then we apply the inverse of the straightening procedure discussed in the Appendix, after first identifying $\RR_+\times \RR$ with $(0,\infty)\times \RR_+\subset\RR_+\times \RR_+$ in a way such that $\{0\}\times \RR_+$ corresponds to $(0,1]\times \{0\}$. 

The result of this procedure, denoted by $\Psi(W)$, defines  a \emph{cornered $\dell$-trivial $h$-cobordism on} $M$, that is, a compact smooth manifold with corners (of codimension 2), neatly embedded into $\corn{n+d}{2}$, such that the inclusion of each of the boundary parts $M$ and $N$ into $\Psi(W)$ is a homotopy equivalence. Again, we call $N \subset \Psi(W)$ the \emph{free boundary part} of the cornered $\partial$-trivial $h$-cobordism $\Psi(W)$.  Let $\Diff(\Psi(W); M)$ denote the topological group of diffeomorphisms of $\Psi(W)$ that are the identity near $M$ and are cylindrical at the collars of the boundary parts.

This procedure defines a surjective map $\tT\to \T$ where $\T$ is a set of cornered $\partial$-trivial $h$-cobordisms on $M$. 

\begin{lem}\label{lem:comparison_of_diff_1}
Two $\dell$-trivial $h$-cobordisms $W$, $W'$ are diffeomorphic relative to $M$ if and only if $\Psi(W)$ and $\Psi(W')$ are diffeomorphic relative to $M$. Furthermore, we have an isomorphism of topological groups
\[\Diff(W;M) \xrightarrow{\cong} \Diff(\Psi(W); M).\]
\end{lem}
\begin{proof}
A diffeomorphism $\varphi\colon W\to W'$ induces a map $\Psi(\varphi)\colon \Psi(W)\to \Psi(W')$ by conjugating with the canonical homeomorphisms $W\to \Psi(W)$ and $W'\to\Psi(W')$. Then $\Psi(\varphi)$ is a diffeomorphism outside the corner region. Furthermore, $\varphi$ is cylindrical at the collars of $\dell W\subset W$ and of $\dell M\subset \dell W$ if and only if $\Psi(\varphi)$ is cylindrical at the collars of the boundary parts. This implies that collared diffeomorphisms $W\to W'$ correspond to collared diffeomorphisms $\Psi(W)\to \Psi(W')$. This shows the first claim. 

The map in the statement is given by the same rule and therefore defines a bijection; it is a homeomorphism by (the argument of) Lemma A.2.
\end{proof}

It follows that $\T$ contains precisely one element from each diffeomorphism class of cornered $\partial$-trivial $h$-cobordisms relative to $M$. Moreover, we have:
\begin{equation}\label{eq:working_model_1}
 |\Hdiff(M)_{\bullet}| \simeq \coprod_{W\in \T} B\Diff(W; M)
\end{equation}
where $W$ denotes a cornered $\partial$-trivial $h$-cobordism in $\T$. This is the model for the $h$-cobordism space that we will use for the comparison with the $h$-cobordism category.

We now discuss our model for an endomorphism space in the $h$-cobordism category $\hcobdno$. Let $M_0 = (M, e_M \colon M \subset \corn{\infty}{1})$ and $N_0 = (N, e_N \colon N \subset \corn{\infty}{1})$ be a pair of objects in $\hcobdno$. Recall that  
\[\hcobdno(M_0, N_0) \simeq \coprod_{W\in \tR} B\Diff(W; M \amalg N)\]
where $\tR$ is a set of $h$-cobordisms $W$ from $M$ to $N$, containing precisely one element from each diffeomorphism class relative to $M$ and $N$. Similarly as above, we may change the smooth structure of such a $W$, by straightening the corners at $\dell N$. The result, denoted by $\Phi(W)$, is a cornered $\dell$-trivial $h$-cobordism on $M$, whose free boundary part $N'$ is the union of $N$ and the horizontal boundary $\dell_h W$. Furthermore, this cornered $\partial$-trivial $h$-cobordism comes equipped with an additional codimension 0 embedding (namely, the inclusion)
\[\iota_N\colon N\to \inter(N')\]
with a collar, that is, a smooth extension to an embedding of $N\cup_{\dell N} \partial(N)\times [0,\eps)$ (this arises from the collar of $\dell N$ in the horizontal boundary part of $W$). 

This embedding $\iota_N$ is an $h$-embedding in the following sense:

\begin{defn}
Let $N$ and $N'$ be smooth compact $(d-1)$-dimensional manifolds. An \emph{$h$-embedding} from $N$ to $N'$ is a smooth embedding $\iota \colon N \to \inter(N')$ such that 
$N' - \inter(\iota(N))$ is an $h$-cobordism from $\partial N'$ to $\partial N$. (In particular, the embedding $\iota$ is a homotopy equivalence.) We denote by 
$\Emb^{\sim}(N, N') \subseteq \Emb(N, N')$ the \emph{space of} $h$-\emph{embeddings} from $N$ to $N'$. 
\end{defn}

A \emph{cornered $\partial$-trivial $h$-cobordism on $M$ with an embedding of $N$} is a pair $(W, \iota_N)$, where $(W; M)$ is a cornered $\partial$-trivial $h$-cobordism on $M$, and $\iota_N$ is a collared $h$-embedding of $N$ into the free boundary part of $W$. The above procedure defines a surjective map
\[\tR\to \R\]
where $\R$ is a set of cornered $\partial$-trivial $h$-cobordisms on $M$ with an embedding of $N$.  

The proof of the following Lemma is similar to the proof of Lemma \ref{lem:comparison_of_diff_1}. 

\begin{lem}\label{lem:comparison_of_diff_2}
Two $h$-cobordisms $W$, $W'$ from $M$ to $N$ are diffeomorphic relative to $M$ and $N$ if and only if there is  a diffeomorphism  $\Phi(W)\to \Phi(W')$ relative to $M$ that sends $\iota_N$ to $\iota'_N$ and is cylindrical  in the collar of $\iota_N$ and $\iota'_N$.  Furthermore, we have an isomorphism of topological groups
\[\Diff(W;M,N) \cong \Diff(\Phi(W); M, \iota_N)\]
where $\Diff(\Phi(W); M, \iota_N)$ is the group of diffeomorphisms of $\Phi(W)$ that are the identity near $M$ and $\iota_N(N)$. 
\end{lem}

As a consequence, $\R$ contains precisely one element from each diffeomorphism class of cornered $\partial$-trivial $h$-cobordisms with an embedding of $N$; where diffeomorphisms in this context are supposed to  be relative to $M$ and $\iota_N$, and cylindrical near $\dell N$. Furthermore,
\begin{equation}\label{eq:working_model_2}
 \hcobdno(M_0, N_0) \simeq \coprod_{(W, \iota_N)\in \R} B\Diff(W; M, \iota_N).
\end{equation}

To compare this model of the morphism space with our model for the $h$-cobordism space, we choose our set $\tR$ in such a way that for each $(W, \iota_N)$ in $\R$, the underlying $W$ 
is contained in $\T$. Then, we have a canonical map which forgets the embedding
\begin{equation*}
\coprod_{(W, \iota_N)\in \R}B\Diff(W; M, \iota_N) \to \coprod_{W\in \T} B\Diff(W;M) 
\end{equation*}
given by the forgetful map $\R \to \T$ and the inclusions $\Diff(W;M, \iota_N)\to \Diff(W;M)$. Under the equivalences in \eqref{eq:working_model_1} and \eqref{eq:working_model_2}, this map corresponds to the map
\[\hcobdno(M,N)\to \Hdiff(M)\]
which was informally described in subsection \ref{outline}.

\subsection{\normalfont \textit{Comparison (without $\theta$-structures)}} \label{subsec:without_theta}
Throughout this subsection, let $(W;M)$ be a cornered $\partial$-trivial $h$-cobordism, whose free boundary part we denote by $N'$, and let $N$ be a compact smooth $(d-1)$-dimensional manifold. 

\smallskip

\noindent We consider the action of $\Diff(W;M)$ on $\Emb^{\sim}(N, N')$ by postcomposition defined in the following way: given $g \in \Diff(W; M)$ and $\iota_N \in \Emb^{\sim}(N, N')$, 
$$g \circ \iota_N = (N \stackrel{\iota_N}{\hookrightarrow} \partial W - M \xrightarrow{g_{|\partial W - M}} \partial W - M).$$
For any $h$-embedding $\iota_N\colon N\to \inter(N')$, restricting the action above to the orbit of the element $\iota_N \in \Emb^{\sim}(N,N')$, we obtain an action map:
\begin{equation} \label{action_map} 
(-\circ \iota_N) \colon \Diff(W; M) \to \Emb^{\sim}(N, N'). 
\end{equation}

\begin{lem} \label{comparison_of_diff_spaces-1}
The action map $(- \circ \iota_N)$ is a fiber bundle and there is a homotopy pullback square:
$$
\xymatrix{ 
\Diff(W; M,  \iota_N) \ar[d] \ar[r] & \Diff(W;M) \ar[d]^{(-\circ \iota_N)} \\
\ast \ar[r]_{\iota_N} & \Emb^{\sim}(N, N').
}
$$
\end{lem}
\begin{proof}
The action map is a fiber bundle projection as a consequence of the isotopy extension theorem (see also \cite{Lima}). There is a canonical homomorphism from 
$\Diff(W;M, \iota_N)$ into the stabilizer $D$ of $\iota_N \in \Emb^{\sim}(N, N')$ with respect to the $\Diff(W;M)$-action. This map is not quite a homeomorphism, because 
the elements in $\Diff(W;M,\iota_N)$ agree with the identity in a collar neighborhood of $\iota_N(N)$. However, it is a homotopy equivalence, where a homotopy inverse is obtained 
by conjugation with a diffeomorphism of $W$ that spreads $\iota_N(N)$ slightly into a collar neighborhood inside $\dell W$. 
\end{proof}

\medskip

\noindent \textbf{Notation.} We recall some standard notation. Given a topological group $G$ and a left (resp. right) $G$-space $X$, we denote the space of homotopy orbits by 
$G \qh X  : = EG \times_G X$ (resp., $X \hq G$). We recall that there is a canonical fiber sequence as follows: $X \to  X\hq G \to  *\hq G = BG$. For a subgroup $H$, the space $G\hq H$ has a left $G$-action, and we have a weak equivalence: $G \qh (G\hq H) \simeq BH$.

\medskip

We denote by $$\Rfib \subset \R$$ the preimage of $W \in \T$ under the forgetful  map $\R \to \T$, that is, $\Rfib$ consists of those $(W', \iota_N)$ with $W'=W$.

\begin{prop} \label{comparison_of_diff_spaces-2}
The action maps $\{(- \circ \iota_N)\}_{(W, \iota_N) \in \Rfib}$  define a weak equivalence of $\Diff(W;M)$-spaces:
$$\coprod_{\Rfib} \Diff(W;M) \hq \Diff(W; M, \iota_N) \xrightarrow{\simeq} \Emb^{\sim}(N, N').$$
\end{prop}
\begin{proof}
The map is given by the action map on each summand, and passing to homotopy orbits; it is clearly $\Diff(W;M)$-equivariant. Let $\Emb^{\sim}(N, N')_{\iota_N}$ denote the path-components of $\Emb^{\sim}(N, N')$ which are in the image of the action map \eqref{action_map}. By Lemma \ref{comparison_of_diff_spaces-1}, we obtain a homotopy fiber sequence:
$$
\coprod_{\Rfib} \Diff(W; M,  \iota_N) \to \coprod_{\Rfib} \Diff(W;M) \xrightarrow{\coprod_{\Rfib}(-\circ \iota_N)}  \coprod_{\Rfib} \Emb^{\sim}(N, N')_{\iota_N}
$$
where the second map is also surjective on $\pi_0$. Therefore the action maps give a weak equivalence
\[\coprod_{\Rfib} \Diff(W;M) \hq \Diff(W; M, \iota_N) \to \coprod_{\Rfib} \Emb^{\sim}(N, N')_{\iota_N}\]
and it suffices to show that the canonical map
$$\coprod_{\Rfib} \Emb^{\sim}(N, N')_{\iota_N} \to \Emb^{\sim}(N, N')$$
is a bijection on $\pi_0$. Surjectivity on $\pi_0$ is clear since $\Rfib$ contains representatives of all orbits of $h$-embeddings $N \to \inter(N')$ as a consequence 
of Lemma \ref{lem:comparison_of_diff_2}. For injectivity, suppose that for two $h$-embeddings 
$\iota_N, \iota'_N \colon N\to  \inter(N')$, and  $g, g' \in \Diff(W;M)$, we have that $g \circ \iota_N$ is isotopic to $g'\circ \iota'_N$. Since \eqref{action_map} is a fiber bundle by Lemma \ref{comparison_of_diff_spaces-1}, it follows that there is $g'' \in \Diff(W;M)$ such that $g''\circ (g \circ \iota_N) = g' \circ \iota'_N$. It follows that $(W; M,\iota_N)$ and $(W; M, \iota'_N)$ are diffeomorphic, and therefore $\iota_N=\iota'_N$, because $\Rfib$ only contains one representative from each diffeomorphism class.  
\end{proof}

\begin{cor} \label{fiber_sequence_1}
There is a diagram: 
$$
\xymatrix{
\coprod_{\Rfib} \Diff(W;M) \hq \Diff(W;M, \iota_N) \ar[d] \ar[r]^(.6){\simeq} & \Emb^{\sim}(N, N') \ar[d] \\
\coprod_{\Rfib} B \Diff(W; M,  \iota_N) \ar[r]^(.45){\simeq} \ar[d] & \Emb^{\sim}(N, N')\hq \Diff(W;M) \ar[d] \\
B \Diff(W;M) \ar@{=}[r]                                       & B\Diff(W;M) 
}
$$
where the vertical rows are the canonical homotopy fiber sequences and the horizontal arrows are homotopy equivalences. 
\end{cor}

\begin{rem} \label{endo_space_model}
Note that the middle horizontal homotopy equivalence in Corollary \ref{fiber_sequence_1} combined with \eqref{eq:working_model_2} yield yet another model for the space 
$\hcobdno(M, N)$, namely
\begin{align*}
\hcobdno(M, N) \simeq \coprod_{(W; M, N') \in \T}  \Emb^{\sim}(N, N')\hq \Diff(W;M).
\end{align*}
Thus, the space of $h$-cobordisms from $M$ to $N$ is equivalent to the space of cornered $\partial$-trivial $h$-cobordisms on $M$, equipped with an additional $h$-embedding of $N$ into the free boundary of the $h$-cobordism.
\end{rem}

\subsection{\normalfont \textit{Comparison (with $\theta$-structures)}} \label{subsec:with_theta}
We now treat the case with $\theta$-structures and obtain an analogue of Corollary \ref{fiber_sequence_1} in this more general case. We fix a vector bundle 
$\theta = (V \to X)$ of rank $d$ over a space $X$. Given a pair of objects in $\hcobdno$,
$$M_0 = (M, e_M \colon M \hookrightarrow \{0\} \times \RR_+ \times \RR^N)$$
$$N_0 = (N, e_N \colon N \hookrightarrow \{0\} \times \RR_+ \times \RR^N),$$
we consider the composite map:
\begin{multline} \label{map_b}
b \colon \hcobdno(M_0, N_0) \simeq \coprod_{\tR} B \Diff(W; M \amalg N)\\ \simeq \coprod_{\tR} \Bun(TW, \theta; l_M)\hq\Diff(W; M \amalg N) 
\to \Bun(\eps \oplus TN, \theta) 
\end{multline}
where $\Bun(TW, \theta; l_M) \simeq \ast$ for every $h$-cobordism $W \in \tR$ and the last map is induced by the restriction map along the canonical bundle map $\eps \oplus TN \to TW$. We will use this map to give 
a description of the morphism spaces in $\hcobd$ in terms of the morphism spaces in $\hcobdno$. We recall that given a pair of objects in $\hcobd$,
$$(M_0, l_M \colon \eps \oplus TM \to \theta)$$
$$(N_0, l_N \colon \eps \oplus TN \to \theta),$$
the space of morphisms $\hcobd(M_0, N_0)$ is identified up to homotopy equivalence with the following space:
\begin{equation} \label{model_for_endo_space_2}
\hcobd((M_0, l_M), (N_0, l_N)) \simeq  \coprod_{\tR} \Bun(TW, \theta; l_M, l_N) \hq \Diff(W; M \amalg N)
\end{equation}
where the coproduct ranges again over the set $\tR$, that is, a set of compact smooth $d$-dimensional (collared) $h$-cobordisms from $M$ to $N$ (with corners), one from each diffeomorphism class.

\begin{lem}\label{lem:theta_no_theta_fiber_sequence}
Let $M_0$ and $N_0$ be objects in $\hcobdno$ as above, and let
\[l_M\colon \eps\oplus TM\to \theta, \quad l_N\colon \eps\oplus TN\to \theta\]
be $\theta$-structrures on $M$ and $N$. Then there is a homotopy pullback square: 
\[
\xymatrix{
\hcobd((M_0, l_M), (N_0, l_N)) \ar[d] \ar[r] & \hcobdno(M_0, N_0) \ar[d]^{b} \\
\ast \ar[r]_{l_N} & \Bun(\eps\oplus TN, \theta)
}
\]
where the top map is the map that forgets the $\theta$-structures.
\end{lem}
\begin{proof}
As explained above, we have the following homotopy equivalences:
\begin{align*}
 \hcobdno(M_0, N_0) &\simeq \coprod_{\tR} *\hq \Diff(W; M \amalg N),\\
 \hcobd((M_0, l_M), (N_0, l_N)) & \simeq \coprod_{\tR} \Bun(TW, \theta; l_M, l_N)\hq \Diff(W; M \amalg N).
\end{align*}
The forgetful map $\hcobd((M_0, l_M), (N_0, l_N)) \to \hcobdno(M_0, N_0)$ is induced by the inclusion
\[\coprod_{\tR} \Bun(TW, \theta; l_M, l_N) \to \coprod_{\tR} \Bun(TW, \theta; l_M) \]
after passing to the homotopy quotients. We recall here that $\Bun(TW, \theta; l_M)\simeq *$, as $W$ is an $h$-cobordism.
Consider the following homotopy pullback square 
$$
\xymatrix{
\coprod_{\tR} \Bun(TW, \theta; l_M, l_N) \ar[r] \ar[d] & \coprod_{\tR} \Bun(TW, \theta; l_M)  \ar[d] \\
\coprod_{\tR} \ast \ar[r]^(.4){\coprod_{\tR} l_N} & \coprod_{\tR} \Bun(\eps \oplus TN, \theta)
}
$$
where the bottom map $l_N$ denotes the constant map at $l_N \in \Bun(\eps \oplus TN, \theta)$. Passing to the homotopy quotients
with respect to the actions of $\Diff(W; \partial_v W)$, we obtain another homotopy pullback square:
$$
\xymatrix{
\hcobd((M_0, l_M), (N_0, l_N))  \ar[r] \ar[d] & \hcobdno(M_0, N_0) \ar[d]^{(\id, b)} \\
\hcobdno(M_0, N_0) \ar[r]^(.35){(\id, l_N)} & \hcobdno(M_0, N_0)\times \Bun(\eps \oplus TN, \theta), 
}$$
and the claim follows since the map $b$ of \eqref{map_b} is the second coordinate of the right vertical map.
\end{proof}

\begin{defn} \label{theta-h-emb}
Let $N$ and $N'$ be compact smooth $(d-1)$-dimensional manifolds, and let 
$$l_N \colon \eps \oplus TN \to \theta, \quad l_{N'} \colon \eps \oplus TN' \to \theta$$
be bundle maps endowing $N$ and $N'$ with $\theta$-structures. We define 
the \emph{space of} $\thetah$-\emph{embeddings} $\Emb^{\sim}_{\theta}(N, N')$ to be the homotopy fiber at $(l_N \colon \eps \oplus TN \to \theta)$ of the map
$$s_{\theta, N} \colon \Emb^{\sim}(N, N') \longrightarrow \Bun(\eps \oplus TN, \theta), \quad (\iota \colon N \hookrightarrow \inter(N')) \mapsto \big(l_{N'} \circ (\id \oplus D\iota)\big).$$
\end{defn}

\noindent \textbf{Proof of Theorem \ref{thm:main1}.} Let $W$ be a $\dell$-trivial $h$-cobordism on $M$, with free end $N'$, and let $N$ be a compact smooth $(d-1)$-manifold with a $\theta$-structure $l_N \colon \eps \oplus TN \to \theta$. We use the notation $M_0$ and $N_0$ as before to indicate the corresponding objects in $\hcobdno$ and denote by $l_M \colon \eps \oplus TM \to \theta$ the $\theta$-structure on $M$ as an object in $\hcobd$. Note that both $W$ and $N'$ inherit $\theta$-structures from $l_M$. Moreover, the triangle
\[\xymatrix{
 \Emb^{\sim}(N, N') \ar[rr]^{\text{Cor.} \ref{fiber_sequence_1} + \text{Rem.} \ref{endo_space_model}} \ar[rd]_{s_{\theta, N}} && \hcobdno(M_0, N_0) \ar[ld]^b\\
 & \Bun(\eps\oplus TN, \theta)
}\]
is commutative up to a preferred homotopy. Passing to the vertical homotopy fibers at $l_N \in \Bun(\eps \oplus TN, \theta)$ and using Lemma \ref{lem:theta_no_theta_fiber_sequence}, we obtain the following diagram in which the upper square is a homotopy pullback:
\[\xymatrix{
 \Emb^{\sim}_\theta(N, N') \ar[r] \ar[d] 
 & \hcobd((M_0, l_M), (N_0, l_N))  \ar[d]
  \\
 \Emb^{\sim}(N, N') \ar[r] \ar[d]
 & \hcobdno(M_0, N_0) \ar[d]
 \\
  \{*\} \ar[r]^W
  & \Hdiff(M).
}\]
The lower square is a homotopy pullback by Corollary \ref{fiber_sequence_1}. Therefore the composite square is also a homotopy pullback. Then the claim follows from Proposition 
\ref{h-cob-cat}. \qed

\subsection{\normalfont \textit{Tangent bundles as $\theta$-structures}}\label{subsec_proof_of_thm2} 
In this subsection, we specialize our previous results to the case where the $\theta$-structure corresponds to the tangent bundle of a compact smooth $(d-1)$-manifold. 

\smallskip

Let $M$ be a smooth connected compact $(d-1)$-manifold of handle dimension 
$k < d - 3$. Let $\theta_M$ denote the vector bundle $\eps \oplus TM$ over 
$M$ of rank $d$. Note that a $\theta_M$-structure on a compact smooth $d$-cobordism $W$ is a bundle map $TW \to \eps \oplus TM.$ 
Clearly $M$ admits a $\theta_M$-structure which is given by the identity map $\id_{\eps \oplus TM}$.

\begin{prop} \label{connectivity-result}
Let $(W; M, N')$ be a cornered $\partial$-trivial $h$-cobordism on $M$ together with a $\theta_M$-structure $l_W \colon TW \to \eps \oplus TM$ which extends the identity structure on $M$. Then the 
space $\Emb^{\sim}_{\theta_M}(M, N')$ is $(d-2k-3)$-connected. 
\end{prop}
\begin{proof}
We first consider the inclusion map $\Emb(M, N') \to \Imm(M, N')$ into the space of immersions. This map
is $(d-2k-2)$-connected by \cite[p.~70]{Weiss(99)},  \cite[Corollary 2.5]{GW}. 
Note that for such homotopical assertions, it suffices to work with the interior of the manifold $M$. By our assumptions on the handle dimension $k < d-3$, it follows that $\partial M$ is also connected and $\pi_1(\partial M) \cong \pi_1(M)$. As a consequence, the space $\Emb^{\sim}(M, N')$ consists exactly of those smooth embeddings which are homotopy equivalences. By restricting to 
the appropriate set of path components, we obtain also a $(d-2k-2)$-connected map $\Emb^{\sim}(M, N') \to \Imm^{\sim}(M,N')$, where $\Imm^{\sim}(M,N')$ is the space of immersions which are also homotopy equivalences. 

By immersion theory and the Hirsch-Smale theorem, the space of immersions $\Imm(M, N')$ is homotopy equivalent to the space $\Bun(TM, TN')$. Note that  $M$ has no closed 
components because it is homotopy equivalent to a space of smaller dimension. As a consequence, we also have $\Imm^{\sim}(M, N') \xrightarrow{\simeq} \Bun^{\sim}(TM, TN')$ where the latter space 
consists of 
those bundles maps which are homotopy equivalences between the bases. 

The space $\Bun(TM, TN')$ includes into the space $\Bun(\eps \oplus TM , \eps \oplus TN')$ by stabilization. This map is $(d-k-2)$-connected since the inclusion 
$O_{d-1} \to O_{d}$ is $(d-2)$-connected and $M$ is $k$-dimensional up to homotopy equivalence. 

Lastly, the map $\Bun(\eps \oplus TM, \eps \oplus TN') \to \Bun(\eps \oplus TM, \eps \oplus TM)$, given by composition with the $\theta_M$-structure on $N'$,
is a homotopy equivalence. This is because this $\theta_M$-structure is a homotopy equivalence between the bases since it is restricted from $l_W$ and $W$ is an 
$h$-cobordism on $M$.

As a consequence, the map $s_{\theta_M, M} \colon \Emb^{\sim}(M, N') \longrightarrow \Bun^{\sim}(\eps \oplus TM, \eps \oplus TM)$
is $(d-2k-2)$-connected. Therefore its homotopy fiber $\Emb^{\sim}_{\theta_M}(M, N')$ is $(d-2k-3)$-connected.
\end{proof}

\noindent \textbf{Proof of Theorem \ref{thm:main2}.} As before, we fix an embedding $e_M \colon M \hookrightarrow \{0\} \times \RR_+ \times \RR^{\infty}$ and consider $(M, e_M, \id_{\eps \oplus TM})$ as an object of $\hcobd$. Then the result follows by combining Proposition \ref{connectivity-result} with Theorem \ref{thm:main1}. \qed

\appendix

\section{Straightening corners} \label{appendix}

The goal of this appendix is to give a rigorous model for the straightening procedure used in this article. 

Let $I:=[0,1]$ and $\RR_+:=[0,\infty)$. Choose, once and for all, a homeomorphism 
\begin{equation}\label{eq:bending_map}
 \Phi\colon I\times \RR_+ \to \RR_+\times \RR_+
\end{equation}
which is a diffeomorphism except at the corner point $(1,0)$, and which is prescribed in a neighborhood of the boundary as follows:
\begin{enumerate}
 \item\label{item:straight1} in a neighborhood of $\{0\} \times \RR_+ $ and of $[0,1)\times \{0\}$, $\Phi$ is the identity;
 \item\label{item:straight2} in a neighborhood of $\{1\}\times (0,\infty)$, $\Phi$ is rotation by $\pi/2$;  \item\label{item:straight3} in a neighborhood of the corner point $(1,0)$, $\Phi$ preserves the radial coordinate and scales the angle by a diffeomorphism
 \[\alpha\colon [0,\pi/2] \to [0,\pi] \]
 where $\alpha(t)=t$ on $[0,\pi/6]$, and $\alpha(t)=t+\pi/2$ on $[\pi/3, \pi/2]$. Here the angle is measured clockwise from (1,0).
\end{enumerate}

\begin{rem}
A more naive condition in (\ref{item:straight3}) would be that $\alpha$ just doubles the  angle. But  this is inconsistent with conditions (\ref{item:straight1}) and (\ref{item:straight2}).
\end{rem}

We will show at the end of this section that such a map  exists, and, moreover, that the space of all such maps, equipped with a suitable topology, is contractible. 

\smallskip

For the ease of notation, we will also write $\Phi$ for any map of the type $\Phi\times \id_{\RR^n}$. With this convention in mind, if $W\subset I\times \RR_+\times \RR^{n-2}$ is a compact neat submanifold, then 
\[\Phi(W)\subset \RR_+\times \RR_+ \times \RR^{n-2}\]
is a again neat submanifold. Note that $\Phi(W)$ and $W$ are homeomorphic. They are diffeomorphic except at the region $\{1\}\times \{0\} \times \RR^{n-2}$, where $W$ has corner points but $\Phi(W)$ does 
not. Moreover, if 
\[e\colon W\to I\times\RR_+\times \RR^{n-2}\]
is a neat embedding, then 
\[\Phi\circ e\circ \Phi\inv\colon \Phi(W)\to \RR_+\times \RR_+\times \RR^{n-2}\]
is also a smooth neat embedding. 

\begin{lem}\label{lem:bending_continuous}
Let $W \subset I \times \RR_+ \times \RR^{n-2}$ be a compact neat $d$-dimensional submanifold (possibly with corners). Then the map
\begin{equation}\label{eq:bending_map_on_embedding}
\Phi\circ - \circ \Phi\inv\colon \Emb(W, I\times \RR_+\times \RR^{n-2}) \to \Emb(\Phi(W), \RR_+\times \RR_+\times \RR^{n-2})
\end{equation}
is continuous.
\end{lem}

We recall that the topology on the spaces of (neat) embeddings under consideration is such that a map from the embedding space is continuous 
if and only if it is continuous on the restriction to the subspace of $\eps$-neat embeddings, for every  $\eps>0$.

\begin{proof}
\newcommand{\halfinter}{\mathrm{int}'}
We consider the subspace
\[\halfinter W:=W\cap ([0,1)\times (0,\infty)\times \RR^{n-2}).\]
For any $\delta>0$, there is an (injective, continuous) restriction map
\[R\colon \Emb^\delta(\Phi(W), \RR_+\times \RR_+\times \RR^{n-2}) \to \Emb(\Phi(\halfinter W), \RR_+\times (0,\infty)\times \RR^{n-2})\]
where both the domain and the target have the $C^\infty$-topology.

\smallskip

\noindent \textit{Claim.} The map $R$ is a topological embedding. 

\smallskip

This claim will imply the Lemma: choosing $\delta>0$ small enough, we have a commutative square
\[\xymatrix{
 \Emb^\eps(W, I\times \RR_+\times \RR^{n-2}) \ar[rr]^R \ar[d]^{\Phi\circ - \circ \Phi\inv} 
  && \Emb(\halfinter W, [0,1)\times (0,\infty)\times \RR^{n-2}) \ar[d]^{\Phi\circ - \circ \Phi\inv} 
  \\
   \Emb^\delta(\Phi(W), \RR_+\times \RR_+\times \RR^{n-2}) \ar[rr]^R 
   && \Emb(\Phi(\halfinter W), \RR_+\times (0,\infty)\times \RR^{n-2}))
}\]
where the right vertical map is continuous, for it is given by pre-composition and post-composition with a smooth map; then the left vertical map is continuous as well, by the definition of the subspace topology.

To prove the claim, we recall that the topology on the domain of $R$ is given by the family of semi-norms
\[\Vert e\Vert_{\beta,\alpha}:= \sup_{x\in U_\beta} \big\Vert \frac{\partial^\alpha (e\circ h_\beta)}{\partial x_{\alpha}} (x) \big\Vert\]
where $\beta$ runs through the indexing set of a cover of $\Phi(W)$ by compact charts 
\[h_\beta\colon \RR_+\times \RR_+\times \RR^{d-2} \supset U_\beta \to V_\beta\subset \Phi(W),\]
and $\alpha$ runs through the set of multi-indices in $\{1, \dots, d\}$. We may choose such a cover as follows: 
the charts that meet a $\delta/2$-neighborhood of $\Phi(W-\halfinter W)$, are of the form $\bar h_\beta\times \id_{[0,\delta/2]}$ for some chart $\bar h_\beta$ on the boundary. 

In contrast, the topology on the target of $R$ is given by a similar family of semi-norms $\Vert e\Vert_{\beta', \alpha}$ where $\beta'$ now runs through the indexing set of a cover 
of $\halfinter \Phi(W)$ by compact charts. Here we use the following cover: the charts that do not meet the $\delta/2$-neighborhood of $\Phi(W-\halfinter W)$ are precisely 
the same as above; for each chart $\bar h_\beta\times \id_{[0,\delta/2]}$ as above, we choose its countable cover by charts $\bar h_\beta\times \id_{[\delta/n, \delta/2]}$. 

Now assume that $e$ is in the image of $R$. Then, for any $\alpha$, the semi-norms associated to such a countable family of charts agree with the seminorm associated to the chart 
$\bar h_\beta\times \id_{[0,\delta/2]}$, since $e$ is $\delta$-neat. This shows that the family of seminorms on the image of $R$ is equivalent to the family of seminorms on the 
domain of $R$.
\end{proof}

Let $\theta:= (V \to  X)$ be a vector bundle of rank $d$ on a space $X$. Next we discuss how $\Phi(W)$ inherits a $\theta$-structure from $W$. The map 
\[\Phi\colon W\to \Phi(W)\]
is a diffeomorphism except at the set corner points, which we denote by $\dell_{01}W$. This induces a bundle isomorphism
\[D\Phi \colon TW\vert_{W-\dell_{01}W} \to T\Phi(W)\vert_{\Phi(W-\dell_{01}W)}\]
covering $\Phi$. This bundle map does not extend to a bundle map on $TW$ (unless $\dell_{01}W$ is empty). Thus, instead of $\theta$-structures on $W$, it will 
be convenient to work here 
with bundle maps
\[l_W\colon TW \vert_{\inter W} \to \theta.\]
Such a bundle map determines, uniquely up to a contractible choice, a bundle map on all of $TW$, which then defines a $\theta$-structure on $W$. As a consequence, 
a $\theta$-structure $l_W$ on $W$ induces (canonically up to a contractible space of choices) a $\theta$-structure on $\Phi(W)$ that corresponds to the bundle map
\[l_{\Phi(W)}\colon T\Phi(W)\vert_{\inter W} \xrightarrow{D\Phi\inv} TW\vert_{\inter W} \xrightarrow{l_W} \theta.\]
We will omit the details of writing down explicitly the corresponding (zigzag) continuous map between the respective spaces of embeddings of manifolds with $\theta$-structures. 

\begin{rem}
As mentioned in the introduction, we can use this procedure on objects and morphisms to define a stabilization functor on the cobordism category. Note however that, on the level of morphisms spaces, we are not allowed to restrict the bundle map to the interior of the manifold, as this would destroy the source-target map. Instead, we can restrict the bundle map to the ``horizontal interior'', that is, 
the intersection with the subset $[0,a] \times \inter{(\corn{\infty-1+n}{1})}\subset [0,a]\times \corn{\infty-1+n}{1}.$
\end{rem}

Next we show that a map $\Phi$ as postulated in \eqref{eq:bending_map} exists. Let $f\colon \RR\to \RR$ be a smooth function such that $f(\varphi)=0$
for $\varphi\leq \pi/6$ and $f(\varphi)=1$ for $\varphi \geq \pi/3$. Let $\alpha\colon \RR\times \RR\to \RR$ be the smooth function
satisfying the differential equation:
\[\frac{d}{dt} \alpha(t,\varphi) = f(\alpha(t,\varphi)), \quad \alpha(0,
\varphi) =  \varphi\]
for all $(t,\varphi) \in \RR \times \RR$. Then $\alpha:=\alpha(\pi/2,-)$ is a function as stipulated in condition (iii) in the definition of $\Phi$.

Let $V\subset \RR_+\times \RR_+$ be a small neighborhood of $[1,\infty)\times \RR_+$. We define a vector field $X$ on $V-\{(1,0)\}$ in polar coordinates 
at $(1,0)$ via 
\[X(r, \varphi) = (0, f(\varphi)).\]
$X$ vanishes on the intersection of $V$ with a neighborhood of $[0,1)\times \{0\}$. Thus, we can extend it smoothly by $0$ to a
neighborhood of $\{0\}\times \RR_+$ and $[0,1)\times \{0\}$ in $\RR_+ \times \RR_+ - \{(1,0)\}$; and then we extend it further 
to a vector field on all of $\RR_+\times \RR_+ - \{(1,0)\}$ in such a way that the slopes of the vectors are bounded above.

Let us follow the flow of a point in $I\times \RR_+ - \{(1,0)\}$ under this vector field. The flow either converges to a sink or it will eventually meet the axis $\{1\}\times \RR_+$. 
When the latter happens, the flow will continue a radial movement with constant speed in clockwise direction, until it reaches the $x$-axis after time $\pi/2$. 
Therefore, the flow of such a point exists at least for $0\leq t\leq \pi/2$.

If we denote by $\Phi_t \colon I\times\RR_+ - \{(1,0)\} \to \RR_+\times \RR_+$ the
resulting flow, then the map $\Phi \colon = \Phi_{\pi/2}$ continuously extends to $I \times \RR_+$ and has the required properties.

\smallskip

If $\Phi$ and $\Phi'$ are two maps as in \eqref{eq:bending_map}, then 
\[H:= \Phi'\circ \Phi\inv\colon \RR_+\times \RR_+\to \RR_+\times \RR_+\]
is a homeomorphism, diffeomorphism except at $(1,0)$, and there is an $\eps>0$ such that the restriction of $H$ to $\RR_+\times [0,\eps)$ is 
as follows:
\begin{enumerate}
 \item In polar coordinates around $(1,0)$, $H$ preserves the radius and rescales the angle by some self-diffeomorphism $\alpha$ of $[0,\pi/2]$ which is the identity on $[0,\pi/6]$ and $[\pi/3, \pi/2]$.
 \item Outside the cone formed by the rays of angle $\pi/6$ and $\pi/3$ starting at $(1,0)$, $H$ is the identity.
\end{enumerate}

Denote by $D^\eps$ the space  of all such self-homeomorphisms of $\RR_+\times \RR_+$, equipped with the locally convex topology induced by the semi-norms
\[ \Vert H\Vert_{K, \alpha}  := \sup_{x\in K} \big\Vert \frac{\partial^\alpha H}{\partial x_{\alpha}} (x)\big\Vert,\]
where $K$ ranges over the compact subsets of the interior of $\RR_+\times \RR_+$, and $\alpha$ ranges over the multi-indices in $\{1,2\}$.  Finally, we let $D=\colim_{\eps\to 0} D^\eps$.

\begin{lem}
$D$ is contractible.
\end{lem}

\begin{proof}
The inclusion $D^{2\eps}\to D^\eps$ is homotopic to a map which lands inside the subspace $D^\eps_\dell$ of those maps where $\alpha=\id$. The homotopy is given by taking the straight line isotopy between $\alpha$ and the identity, and applying it in the radial coordinate around $(1,0)$. 

Now note that any element of $D^\eps_\dell$ defines just as well a diffeomorphism of the manifold $K$ obtained from $\RR_+ \times \RR_+$ by smoothing the corners.
In more detail, let $K$ be the manifold with boundary (but without corners) obtained from $\RR_+\times \RR_+$ by straightening the corners, and let $\Diff^\eps_\dell(K)$ denote the space of diffeomorphisms of $K$ which are the identity in an $\eps$-neighborhood of the boundary,
endowed with the $C^\infty$-topology. Then the inclusion of $D^\eps_\dell$ into the colimit factors canonically through some $\Diff^{\eps'}_\dell(K)$. As $K$ is diffeomorphic to $\RR_+\times \RR$, we have $\Diff^\eps_\dell(K)\simeq *$, by the usual rescaling isotopy. 

Hence the  inclusion $D^{2\eps}\to D$ is nullhomotopic, for any $\eps$. 
\end{proof}

\end{document}